%% file: main.tex
\documentclass[11pt, reqno]{amsart}
\usepackage{format}

\input{author}
\title{\sffamily Connecting the space of marked groups and the space of group operations}
\begin{document}

\begin{abstract}
We establish a connection between two well-studied spaces of countable groups: the space of group operations and the space of marked groups. This connection shows that the two spaces are equivalent in terms of generic properties in the sense of Baire category, which allows us to translate several results from the former setting to the latter. As an application, we give new, shorter proofs of two theorems that concern the generic behavior of compact metrizable abelian~groups.
\end{abstract}

\maketitle    
    
\section{Introduction}\label{s.intro}

There are multiple ways to equip the class of countable groups with a nice topology. This paper focuses on two of these and the connection between them.

The first approach is to exclude finite groups, fix $\nat$ as a common underlying set, and consider the set of all group operations on $\nat$ as a subspace of $\nnn$, which carries the product of the discrete topologies. We call the resulting space $\calg$ \emphd{the space of group operations}. This approach stems from the general concept of so-called logic spaces (see, e.g, \cite[Def~16.5]{kechris2012classical}) and has been used in \cite{ivanov1999generic, goldbring2023generic, elekes2024generic, darji2023generic, ivanov2024generic} to study generic properties of groups in the sense of Baire category. It is very natural from a logician's perspective.

The second approach captures countable groups as quotients of the free group $F_\infty$ of rank $\omega$. In this case, groups are encoded by normal subgroups of $F_\infty$, which form a closed subspace $\calm$ in $2^{F_\infty}$ called \emphd{the space of marked groups}. The concept was sketched by Gromov in his celebrated paper \cite{gromov1981groups} on polynomial growth and was made explicit by Grigorchuk in another influential paper \cite{grigorchuk1984degrees}. Since then, spaces of marked groups have been extensively studied \cite{champetier2000Lespace, champetier2004limit, cornulier2009isolated, zheng2022asymptotic}, although primarily not from the perspective of Baire category, with the exception of \cite{osin2021topological}. This setting is more natural from the viewpoint of geometric group theory.

The goal of this paper is to establish a connection between these two spaces that allows the transfer of generic properties. Our proof is direct and requires only elementary tools. It was brought to our attention by A.~Tserunyan that such a connection can also be derived from recent results of R.~Chen \cite{CHEN_2025}. Since this general approach requires more theoretical background, we present only a relatively self-contained argument, also due to R.~Chen, which provides an alternative proof for Corollary~\ref{c.main_intro}. See the appendix.

%To the best of our knowledge no connection has been established between the two settings that would allow the transfer of generic properties.
Some interplay between the two settings has already appeared in \cite{goldbring2023generic}, where Proposition~3.2.1 asserts that the natural map from $\calg$ to $\calm$ (defined in Section~\ref{s.results}) is continuous and surjective. However, this claim, as stated, is incorrect. (This does not affect the main results of \cite{goldbring2023generic}.) In fact, we will show in Section~\ref{s.results} that the range of this map is nowhere dense in $\calm$.

The main result of this paper is the following theorem.

\begin{theorem}\label{t.main_intro}
There is a comeager set $\cald\subseteq\calm$ and a surjective open continuous map $f:\cald\to\calg$ that maps each group to an isomorphic copy of itself. That is, $F_\infty/N\cong(\nat,f(N))$ for every $N\in\cald$.
\end{theorem}

\begin{cor}\label{c.main_intro}
A group property is generic in $\calg$ if and only if it is generic in $\calm$.
\end{cor}

This corollary and a more general variant of it allow us to translate a large number of genericity results from $\calg$ to $\calm$. They also enable us to give new, shorter proofs of the following theorems.

\begin{theorem}\label{t.odometer_intro}
\cite[Thm~4.7]{elekes2024generic}
The isomorphism class of the group
$$Z=\prod_{p\text{ prime}}(Z_p)^\nat,$$
where $Z_p$ denotes the group of $p$-adic integers, is comeager in the space of compact metrizable abelian groups.
\end{theorem}

\begin{theorem}\label{t.solenoid_intro}
\cite[Thm~3.7]{darji2023generic}
The isomorphism class of the universal solenoid is comeager in the space of connected compact metrizable abelian groups.
\end{theorem}

\emphd{Paper outline.} Section~\ref{s.prelim} contains essential preliminaries. Section~\ref{s.results} is dedicated to the results and their direct corollaries. Section~\ref{s.applications} concludes the paper with the above-mentioned applications.

\section{Preliminaries}\label{s.prelim}

We start by introducing notations and conventions.

We denote by $F_n$ the free group generated by the set $X_n=\{x_0,\ldots,x_{n-1}\}$. Similarly, $F_\infty$ denotes the free group generated by the set $X_\infty=\{x_0,x_1,\ldots\}$. For any group $G$, a map from $X_\infty$ (resp.~$X_n$) onto a generating set of $G$ is called a \emphd{marking} of $G$, which extends to a unique homomorphism $\wtilde m:F_\infty\to G$ (resp.~$\wtilde m:F_n\to G$). We think of elements of free groups as reduced words. We may write a word $w$ as $w(x_0,\ldots,x_{n-1})$ to express that the letters of $w$ are \emph{among} $x_0,\ldots,x_{n-1}$. For any group $G$ and elements $a_0,\ldots,a_{n-1}\in G$, when convenient, we simply write $w(a_0,\ldots,a_{n-1})$ for $\wtilde m(w(x_0,\ldots,x_{n-1}))$, where $m$ is any marking that maps $x_i$ to $a_i$ for each $i<n$. We fix an enumeration $F_\infty=\{w_0,w_1,\ldots\}$. 

\subsection{Baire category}\label{ss.baire_category}

Recall the following well-known facts.

\begin{theorem}\label{t.comp_metriz}
\cite[Thm~3.11]{kechris2012classical} A subspace $A$ of a completely metrizable space $X$ is completely metrizable if and only if $A$ is $G_\delta$ in $X$.
\end{theorem}

\begin{theorem}[Baire Category Theorem]
\label{t.BCT}
In a completely metrizable space $X$, nonempty open subsets are non-meager.
\end{theorem}

In a space $X$ of mathematical objects, we can view properties of objects as subsets of $X$. A property is called \emphd{generic} if it is comeager as a subset of $X$. It is easy to prove the following well-known corollary of the Baire Category Theorem.

\begin{prop}\label{p.comeager_dense_G_delta}
A subset $A$ of a completely metrizable space $X$ is comeager if and only if $A$ contains a dense $G_\delta$ subset of $X$.
\end{prop}

We will heavily use the well-known fact that genericity can be transferred back and forth between two Polish spaces via a surjective open continuous map. For completeness, we give a short proof.

\begin{prop}\label{p.open_cont_surj}
For any Polish spaces $X$ and $Y$, subsets $A\subseteq X$ and $B\subseteq Y$, and surjective open continuous map $f:X\to Y$, the following hold:
\begin{enumerate}
    \item $B\subseteq Y$ is comeager $\implies$ $f^{-1}(B)\subseteq X$ is comeager,
    \item $A\subseteq X$ is comeager $\implies$ $f(A)\subseteq Y$ is comeager.
\end{enumerate}
\end{prop}

\begin{proof}
By Proposition~\ref{p.comeager_dense_G_delta}, we may assume that $A$ and $B$ are dense $G_\delta$ in $X$ and $Y$ respectively.

(1) Since $f$ is open, $f^{-1}(B)$ is dense in $X$. Since $f$ is continuous, $f^{-1}(B)$ is $G_\delta$ in $X$.

(2) Note that $f(A)$ has the Baire property since it is analytic (see \cite[Thm~21.6]{kechris2012classical}). Thus, if $f(A)$ is not comeager in $Y$, then there is nonempty open $V\subseteq Y$ in which it is meager, hence $A\cap f^{-1}(V)\subseteq f^{-1}(f(A)\cap V)$ is meager in the open set $f^{-1}(V)$ by (1). This contradicts that $A$ is comeager in $X$ since $f^{-1}(V)$ is nonempty by the surjectivity of $f$.
\end{proof}

\subsection{The space of group operations}\label{ss.group_operations}

We call the set
$$\calg=\left\{G\in\nnn:\ G\text{ is a group operation}\right\}$$
equipped with the topology inherited from the Polish space $\nnn$ (which carries the product of the discrete topologies) \emphd{the space of group operations}. We denote by $\wtilde G$ the group $(\nat,G)$ associated to the group operation $G$. By a \emphd{group property}, we mean an isomorphism invariant subset of $\calg$.

The above definition of $\calg$ differs slightly from the one given in \cite[Subsec~3.1]{elekes2024generic} and will serve us better in this paper. Let us write $\calg$ as the disjoint union of the clopen subsets
$$\calg_n=\left\{G\in\calg:\ \text{the identity element of } \wtilde G \text{ is }n\right\}=\{G\in\calg:\ G(n,n)=n\}$$
with $n\in\nat$. Also note that for any fixed $n,k\in\nat$ the bijection $\varphi:\nat\to\nat$ that swaps $n$ and $k$ and fixes all the other numbers induces a homeomorphism $h_\varphi:\calg_n\to\calg_k$ defined by $h_\varphi(G)(i,j)=\varphi(G(\varphi^{-1}(i),\varphi^{-1}(j)))$. Then for any $G\in\calg_n$, we have $\wtilde G\cong \wtilde{h_\varphi(G)}$ since $\varphi$ itself is an isomorphism between $\wtilde G$ and $\wtilde {h_\varphi(G)}$. Thus, we may view $\calg$ as infinitely many copies of any given $\calg_n$ placed discretely next to each other. In particular, since each $\calg_n$ is $G_\delta$ in $\nnn$ (see, e.g., \cite[Prop~3.1]{elekes2021generic}), we conclude that $\calg$ is also $G_\delta$ in $\nnn$, hence it is Polish with the subspace topology. It is also clear from the above that a group property $\calp$ is comeager in $\calg$ if and only if $\calp\cap\calg_n$ is comeager in $\calg_n$ for some (equivalently any) $n\in\nat$. Thus, in terms of generic group properties, the spaces $\calg$ and $\calg_n$ are equivalent for any $n\in\nat$.

\begin{remark}
In set definitions such as $\{G\in\calg:\ G(a,b)=c\}$, we write $ab$ instead of $G(a,b)$ and use inverses to avoid cumbersome notation. We denote by $e_G$ the identity element of $\wtilde G$. For example, we write
$$\left\{G\in\calg:\ aba^{-1}b^{-1}=e_G\right\}$$
instead of the horror
$$\{G\in\calg:\ \exists x,y,e\in\nat\ (G(e,e)=e\land G(a,x)=e\land G(b,y)=e\land G(G(G(a,b),x),y)=e)\}.$$
\end{remark}

The following simple fact is worth recording.

\begin{prop}\label{p.group_op_basis}
\cite[Prop~3.6]{elekes2024generic}
For any $n\in\nat$ and finite sets $\{u_0,\ldots,u_{k-1}\}$ and $\{v_0,\ldots,v_{l-1}\}$ of words in the variables $x_0,\ldots,x_{n-1}$ and numbers $a_0,\ldots,a_{n-1},b_0,\ldots,b_{k-1},c_0,\ldots,c_{l-1}\in\nat$~the~set
%$$\left\{G\in\calg:\ \left(\bigwedge_{i=0}^{k-1} u_i(a_0,\ldots,a_{n-1})=b_i\right)\land\left(\bigwedge_{j=0}^{l-1}v_j(a_0,\ldots,a_{n-1})\neq c_j\right)\right\}$$
$$\left\{G\in\calg:\ \forall i<k\ (u_i(a_0,\ldots,a_{n-1})=b_i)\text{ and }\forall j<l\ (v_j(a_0,\ldots,a_{n-1})\neq c_j)\right\}$$

is clopen, and sets of this form constitute a basis for $\calg$.
\end{prop}

\begin{remark}\label{r.GKL_space}
A variant $\calg'$ of the space of group operations was defined in \cite{goldbring2023generic} as the set of all triples $(\mu,\iota,e)\in\nnn\times\nat^\nat\times\nat$ that satisfy the group axioms when interpreted as the multiplication map, the inversion map, and the identity element on the underlying set $\nat$, respectively.\footnote{This space was denoted by $\calg$ in \cite{goldbring2023generic}, which we replaced by $\calg'$ to avoid conflicting notation.} By simply reading off the inversion map and the identity element from the multiplication map, we obtain a natural bijection $\calg\to\calg'$, $G\mapsto (G,\iota_G,e_G)$. It is easy to check that this is a homeomorphism.
\end{remark}

In the rest of the paper, we use the equivalence of $\calg$, $\calg'$, and $\calg_n$ (for any $n\in\nat$) implicitly by treating generic group properties in any isomorphism-invariant subspace of any of the spaces $\calg$, $\calg'$, and $\calg_n$ as a generic group property in each of the corresponding subspaces.

\subsection{The space of marked groups}
\label{ss.marked_groups}

As mentioned above, the space of marked groups leverages the simple fact that every countable group is isomorphic to a quotient of $F_\infty$. Let us consider
$$\calm=\{N\subseteq F_\infty:\ N\nsub F_\infty\}$$
as a subspace of the Cantor space $2^{F_\infty}$.
It is straightforward to check that $\calm$ is closed in $2^{F_\infty}$, hence compact. We call $\calm$ \emphd{the space of marked groups}. In $\calm$, a \emphd{group property} is an isomorphism-invariant subset $\calp\subseteq\calm$ in the sense that $M\in\calp$ and $F_\infty/M\cong F_\infty/N$ implies $N\in\calp$ for any $N,M\in\calm$.

\begin{notation}\label{n.associated_group_prop}
For a group property $\calp\subseteq\calg$, we denote by $\calp^*$ the corresponding group property in $\calm$, that is,
$$\calp^*=\left\{N\in\calm:\ \exists G\in\calp\ \left(\wtilde G\cong F_\infty/N\right)\right\}.$$
\end{notation}

\begin{remark}\label{r.marked_basis}
Note that sets of the form
$$\calu=\{N\in\calm:\ u_0,\ldots,u_{k-1}\in N\text{ and } v_0,\ldots,v_{l-1}\notin N\}$$
with $u_0,\ldots,u_{k-1},v_{0},\ldots,v_{l-1}\in F_\infty$ and $k,l\in\nat$ constitute a clopen basis for $\calm$.
\end{remark}

\begin{remark}\label{r.fin_gen_marked}
Spaces of marked groups are usually introduced to study finitely generated groups, in which case $F_\infty$ is replaced by $F_n$ for some $n\in\nat$. Since we are interested in not finitely generated groups as well, we use $F_\infty$.
\end{remark}

\subsection{Two spaces of marked abelian groups}\label{ss.marked_abelian}

There are two natural ways to define the space of marked abelian groups. We can view countable abelian groups as quotients of the free abelian group $F=\bigoplus_{i\in\nat}\integer$, which gives us
$$\cala''=\{N\subseteq F:\ N\leq F\}$$
as a closed subset of the Cantor space $2^F$. The other option is to use
$$\cala'=\{N\in\calm:\ F_\infty/N\text{ is abelian}\},$$
which is closed in $\calm$. The following simple proposition shows that they are equivalent. Let $\psi:F_\infty\to F$ be the unique homomorphism that maps $x_i$ to $e_i$ for every $i\in\nat$, where $e_i(i)=1$ and $e_i(j)=0$ if $i\neq j$.
\begin{prop}\label{p.abelian_equiv}
The map $\Psi:\cala''\to\cala'$, $\Psi(N)=\psi^{-1}(N)$ is a homeomorphism and $F/N\cong F_\infty/\psi^{-1}(N)$ holds for every $N\in\cala''$.
\end{prop}

\begin{proof}
The isomorphism $F/N\cong F_\infty/\psi^{-1}(N)$ follows from the third isomorphism theorem, which also shows that $\Psi$ indeed maps to $\cala'$. Injectivity is clear from the surjectivity of $\psi$. For surjectivity, observe that $\ker\psi$ is the commmutator subgroup of $F_\infty$, hence every $N\in\cala'$ is union of cosets of $\ker\psi$, which implies $\Psi(\psi(N))=\psi^{-1}(\psi(N))=N$.

Since $\cala''$ is compact, it remains to check that $\Psi$ is continuous. But this is clear since $\Psi$ maps subbasic open sets onto subbasic open sets: $\psi(w)\in N\iff w\in\psi^{-1}(N)=\Psi(N)$.
\end{proof}

\subsection{The space of compact metrizable abelian groups}\label{ss.compact_metriz_ab}

For any topological space $X$, let $\calk(X)$ denote the set of all nonempty compact subsets of $X$. The topology on $\calk(X)$ generated by sets of the form
$$\{K\in \calk(X):\ K\subseteq U\}\quad\text{and}\quad\{K\in \calk(X):\ K\cap V\neq\emptyset\}$$
with $U,V\subseteq X$ open is called the \emphd{Vietoris topology}. It is well-known that $\calk(X)$ inherits several topological properties of $X$. For example, if $X$ is compact metrizable, then $\calk(X)$ is also compact metrizable \cite[Thm~4.26]{kechris2012classical}. It is easy to prove that for any topological group $G$, the set
$$\cals(G)=\{K\in\calk(G):\ K\text{ is a subgroup of }G\}$$
is closed in $\calk(G)$ (see \cite[Prop~2.4]{elekes2024generic}).

Let $\circle$ denote the circle group. It is well-known that every compact metrizable abelian group can be embedded into $\circle^\nat$ (see, for example, the beginning of Section~4.2 in \cite{elekes2024generic}). Thus, the compact metrizable space $\cals(\TT)$ can be viewed as \emphd{the space of compact metrizable abelian groups}. It is also easy to check that the subspace
$$\calc(\TT)=\left\{K\in\cals(\TT):\ K\text{ is connected}\right\}$$
is closed in $\cals(\TT)$. Thus, we may view the compact metrizable space $\calc(\TT)$ as \emphd{the space of connected compact metrizable abelian groups}.

\subsection{Pontryagin duality}

The \emphd{dual group} $\what G$ of a locally compact abelian (LCA) group $G$ is the set of all continuous homomorphisms from $G$ to $\circle$ with pointwise multiplication and the compact-open topology. It is well-known that $\what G$ is also an LCA group, and $\widehatto{G}{\widehat G}\cong G$ holds for any LCA group $G$. For a closed subgroup $H$ of an LCA group $G$, the \emphd{annihilator} of $H$ is
$$\ann(H)=\left\{\chi\in\what G:\ \chi|_H\equiv 0\right\},$$
and $\what H\cong\what G/\ann (H)$ holds. See \cite{rudin1962fourier} for the proofs. We will also need the following facts.

\begin{prop}\label{p.pontryagin}
The following hold.
\begin{enumerate}
    \item \cite[Thm~2.2.3]{rudin1962fourier} For any sequence $(A_i)_{i\in\nat}$ of compact abelian groups, $\ds\what{\prod_{i\in\nat}A_i}=\bigoplus_{i\in\nat}\what A_i$.
    \item \cite[Subsec 25.2]{hewitt1979abstract} The dual of the Prüfer $p$-group $\integer[p^\infty]$ is the group $Z_p$ of $p$-adic integers.
    \item The dual of $\ds\bigoplus_{i\in\nat}(\rat/\integer)$, which can be written as $\ds\bigoplus_{p\text{ prime}}\bigoplus_{i\in\nat}\integer[p^\infty]$, is $\ds\prod_{p\text{ prime}}(Z_p)^\nat$. (This follows from (1) and (2).)
    \item \cite[Subsec 25.4]{hewitt1979abstract} The dual of $(\rat,+)$ is the universal solenoid.
    \item \cite[Cor~8.5]{hofmann2013structure} An LCA group $G$ is discrete and torsion-free if and only if $\what G$ is compact and connected.
\end{enumerate}
\end{prop}

The following theorem is a special case of \cite[Thm~14]{fisher2009space}. It establishes a beautiful connection between the spaces $\cals(\TT)$ and $\cala''$ via Pontryagin duality.

\begin{theorem}\label{t.pontryagin_homeo}
\cite[Thm~14]{fisher2009space} The annihilator map $\ann:\cals(\TT)\to\cala''$, $K\mapsto \ann(K)$ is a homeomorphism, and $F/\ann(K)\cong \what K$ for every $K\in\cals(\TT)$. In particular, since Pontryagin duality preserves isomorphism, so does $\ann$.
\end{theorem}

\section{Results}\label{s.results}

In this section, we describe two connections between the space $\calg$ of group operations and the space $\calm$ of marked groups. The latter --- Theorem~\ref{t.main} --- is the main result of the paper.

\begin{defi}
For a group property $\cals\subseteq\calm$, let $(\star)_\cals$ denote the following condition.
\begin{quote}
Every countable marked group with property $\cals$ embeds into an infinite marked group with property $\cals$. That is, for every $M\in\cals$ there is $N\in\cals$ such that $F_\infty/N$ is infinite and $F_\infty/M$ embeds into $F_\infty/N$.
\end{quote}
\end{defi}

Notice that $(\star)_\cals$ is an extremely weak condition. We will need the following lemma.

\begin{lemma}\label{l.open_dense}
Let $\cals\subseteq\calm$ be a group property for which $(\star)_\cals$ holds. For every $m\in\nat$ let
$$\cald_m=\{N\in\calm:\ |F_\infty/N|\geq m \text{ and }\forall i<m\ (|w_iN\cap X_\infty|\geq m)\}$$
and let $\cald=\bigcap_{m\in\nat}\cald_m$.
Then for every $m\in\nat$ the set $\cald_m\cap\cals$ is dense open in $\cals$. In particular, $\cald\cap\cals$ is comeager in $\cals$.
\end{lemma}

\begin{proof}
Fix any $m\in\nat$ and let $[\nat]^m$ denote the set of all $n$-element subsets of $\nat$. We write $\cald_m$ as
$$\cald_m=\bigcup_{i_0,\ldots,i_{m-1}
\in\nat}\ \bigcup_{\left\{j_0^0,\ldots,j_{m-1}^0\right\},\ldots,\left\{j_0^{m-1},\ldots,j_{m-1}^{m-1}\right\}\in[\nat]^m}\ \bigcap_{s,t<m}\left\{N\in\calm:\ w_{i_s}w_{i_t}^{-1}\notin N\text{ and } x_{j_t^s} w_s^{-1}\in N\right\},$$
which shows that it is open. Thus it suffices to prove that $\cald_m$ is dense in $\cals$.

Fix any nonempty basic clopen set
$$\calu=\{N\in\cals:\ u_0,\ldots,u_{k-1}\in N\text{ and } v_0,\ldots,v_{l-1}\notin N\}$$
and any $N_0\in\calu$. Let $n\in\nat$ be such that $X_n=\{x_0,\ldots,x_{n-1}\}$ contains all letters occurring in $u_0,\ldots,u_{k-1}, v_0,\ldots,v_{l-1}$, let $M=N_0\cap F_n$, and let $\rho:F_n\to F_n/M$ denote the quotient map. By $(\star)_\cals$, there is a normal subgroup $N_0'\in\cals$ such that $F_\infty/N_0$ embeds into $H=F_\infty/N_0'$ and $H$ is infinite. Since $F_n/M$ is isomorphic to a subgroup of $F_\infty/N_0$, there is an embedding $\alpha:F_n/M\to H$. This defines a marking $\mu:X_n\to \alpha(F_n/M)$. Let us mark the elements of $H$ with the remaining generators $x_n,x_{n+1},\ldots$ so that every element is marked by at least $m$ many generators. This marking $X_\infty\to H$ extends to a homomorphism $\varphi:F_\infty\to H$. Let $L=\ker\varphi$. It suffices to verify that $L\in\cald_m\cap\calu.$

First, $L\in\cald_m$ is clear from the definition of $\varphi$ and the fact that $H$ is infinite. Since $\cals$ is isomorphism invariant, $L\in\cals$ follows from $F_\infty/L\cong H\cong F_\infty/N_0'$ and $N_0'\in\cals$. For $L\in\calu$, note that $L\cap F_n=\ker(\varphi|_{F_n})=\ker(\alpha\circ\rho)=M$, where the second equality holds because $\varphi|_{F_n}$ and $\alpha\circ\rho$ extend the same marking $\mu$.
\end{proof}

Observe that $\calg$ and $\calm$ are not homeomorphic since $\calm$ is compact and $\calg$ is not. (Even the subspaces $\calg_n$ are non-compact.)

However, we do have a natural map $$\Phi:\calg\to\calm,\ \Phi(G)=\{w(x_0,\ldots,x_{n-1})\in F_\infty:\ w(0,\ldots,n-1)=e_G\}.$$
Note that $\Phi(G)\in\calm$ since $\Phi(G)$ is the kernel of the homomorphism $\varphi_G:F_\infty\to\wtilde G$ that extends the marking $x_i\mapsto i$.
The analogous map for $\calg'$ was defined in \cite{goldbring2023generic}, where Proposition~3.2.1 asserts that it is continuous and surjective. This is not true as the following proposition and Remark~\ref{r.GKL_space} show.

\begin{prop}\label{p.embedding}
For the map $\Phi:\calg\to\calm$ defined above, the following hold.
\begin{enumerate}
    \item[(1)] It is an embedding.
    \item[(2)] We have $\wtilde G\cong F_\infty/\Phi(G)$ for every $G\in\calg$.
    \item[(3)] $\ran(\Phi)=\{N\in\calm:\ \text{each coset in $F_\infty/N$ contains exactly one of the generators $x_0,x_1,\ldots$}\}$.
    \item[(4)] The set $\ran(\Phi)$ is nowhere dense in $\calm$.
\end{enumerate}
In particular, the map $\Phi$ cannot transfer generic properties between $\calg$ and $\calm$.
\end{prop}

\begin{proof}
(1) \emphd{Injectivity.} If $i\cdot j=k$ holds in $\wtilde G$ but not in $\wtilde H$, then $x_ix_jx_k^{-1}\in\Phi(G)\setminus\Phi(H)$.

\emphd{Continuity and openness.} Note that for any $k\in\nat$, $w(x_0,\ldots,x_{n-1})\in F_\infty$ and $G\in\calg_k$, we have $w(x_0,\ldots,x_{n-1})\in \Phi(G)$ if and only if $w(0,\ldots,n-1)=k$ holds in $\wtilde G$. By the injectivity of $\Phi$, Proposition~\ref{p.group_op_basis}, and Remark~\ref{r.marked_basis}, it follows that $\Phi$ is continuous and open.

(2) and (3) follow easily from the observation above that $\Phi(G)=\ker\varphi_G$ for every $G\in\calg$.
%For any $G\in\calg$ the map $\alpha_G: \wtilde G\to F_\infty/\Phi(G)$, $i\mapsto x_i\Phi(G)$ is an injection that preserves relations back and forth. Since $\ran(\alpha_G)$ generates $F_\infty/\Phi(G)$, it is surjective as well.

(4) This follows from (3) and Lemma~\ref{l.open_dense}. 
\end{proof}

%The following observation will be useful in the proof of the main theorem.

%\begin{obs}\label{o.special_basis}
%Sets of the form
%$$\calu=\{N\in\cald:\ u_0,\ldots,u_{k-1}\in N\land v_0,\ldots,v_{l-1}\notin N\}$$
%constitute a clopen basis in $\cald$ even if we additionally require that
%\begin{itemize}
%    \item[(A)] $x_i\in N$ occurs in the definition of $\calu$ for some $i\in\nat$,
%    \item[(B)] there is $K\in\nat$ such that $\{0,\ldots,K-1\}=\{i\in\nat:\ x_i\text{ occurs in }u_0,\ldots,u_{k-1},v_0,\ldots,v_{l-1}\}$,
    %\item[(C)] for every $i<j<K$ the word $x_ix_j^{-1}$ is among $u_0,\ldots,u_{k-1},v_0,\ldots,v_{l-1}$.
%\end{itemize}
%\end{obs}

We turn to the main theorem. Recall from Lemma~\ref{l.open_dense} that $\cald$ is a dense $G_\delta$ set in $\calm$. 

\begin{theorem}\label{t.main}
There is a surjective open continuous map $f:\cald\to\calg$ such that $F_\infty/N\cong\wtilde{f(N)}$ for every $N\in\cald$.
\end{theorem}

\begin{proof}
For any $N\in\cald$ we define $f(N)$ as follows. We choose the least-indexed generator from each coset of $N$ and enumerate these chosen generators in increasing order of their indices: $x_{i_0},x_{i_1},\ldots$. We call $\{x_{i_0},x_{i_1},\ldots\}$ the (enumerated) transversal associated to $N$. Let 
$$f(N)(a,b)=c\iff x_{i_a}N\cdot x_{i_b}N=x_{i_c}N$$
for every $a,b,c\in\nat$. Now $F_\infty/N\cong \wtilde{f(N)}$ is clear from the definition.

\emphd{Surjectivity.} Let $\sigma:X_\infty\to\nat$ be a surjective map so that $\sigma^{-1}(n)$ is infinite for every $n\in\nat$ and $\gamma:\ \nat\to \nat$, $\gamma(n)=\min\{k\in\nat:\ x_k\in\sigma^{-1}(n)\}$ is order-preserving. For any $G\in\calg$, the map $\sigma$ extends to a homomorphism $\sigma_G:\ F_\infty\to\wtilde G$. It is clear from the definition that $f(\ker\sigma_G)=G$.

\emphd{Continuity.} Fix any subbasic clopen set $\calb=\{G\in\calg:\ a\cdot b=c\}$ in $\calg$ and normal subgroup $N_0\in f^{-1}(\calb)$. Let $\{x_{i_0},x_{i_1},\ldots\}$ be the transversal associated to $N_0$. Let $n=\max\{i_a,i_b,i_c\}$. Then
$$\calw=\left\{N\in\cald:\ x_{i_a}x_{i_b}x_{i_c}^{-1}\in N\right\}\cap\bigcap_{i<j\leq n}\left\{N\in\cald:\ x_ix_j^{-1}\in N\iff x_ix_j^{-1}\in N_0\right\}$$
is a neighborhood of $N_0$ that lies in $f^{-1}(\calb)$ since for any $N\in \calw$ the enumerated transversal associated to $N$ coincides with that of $N_0$ up to $x_n$ and $x_{i_a}N\cdot x_{i_b}N=x_{i_c}N$ holds in $F_\infty/N$.

\emphd{Openness.} Fix any nonempty basic clopen set
$$\calu=\{N\in\cald:\ u_0,\ldots,u_{k-1}\in N\text{ and } v_0,\ldots,v_{l-1}\notin N\}.$$
Fix any $N_0\in\calu$ and let $\{x_{i_0},x_{i_1},\ldots\}$ be the transversal associated to $N_0$. Let $\alpha_{N_0}:X_\infty\to\nat$ be defined by $\alpha_{N_0}(x_i)=j\iff x_i\in x_{i_j}N_0$. Let $u_0',\ldots,u_{k-1}',v_0',\ldots,v_{l-1}'$ denote the words obtained by replacing every letter $x_i$ occurring in $u_0,\ldots,u_{k-1},v_0,\ldots,v_{l-1}$ with $\alpha_{N_0}(x_i)$, and let $i_e=\min\{i:\ x_i\in N_0\}$. It suffices to prove the following.

\emph{Claim.} The set
$$\calv=\{G\in\calg:\ i_e\cdot i_e=i_e\text{ and }\forall s<k\ (u_s'=i_e)\text{ and }\forall t<l\ (v_t'\neq i_e)\}$$
is a neighborhood of $f(N_0)$ that lies in $f(\calu)$.

\emph{Proof.} First, $\calv$ is open by Proposition~\ref{p.group_op_basis}, and $f(N_0)\in\calv$ follows from the definitions. To prove $\calv\subseteq f(\calu)$ fix any $H\in\calv$. We need to find $M\in\calu$ with $f(M)=H$.

Let $K\in\nat$ be such that $\{x_0,\ldots,x_{K-1}\}$ contains all letters occurring in $u_0,\ldots,u_{k-1},v_0,\ldots,v_{l-1}$. It is easy to see that there is a surjective map $\varphi:X_\infty\to\nat$ such that
\begin{itemize}
    \item[(1)] it extends $\alpha_{N_0}|_{\{x_0,\ldots,x_{K-1}\}}$,
    \item[(2)] the preimage $\varphi^{-1}(p)$ is infinite for every $p\in\nat$,
    \item[(3)] $p<q\implies\min\{i:\ x_i\in\varphi^{-1}(p)\}<\min\{i:\ x_i\in\varphi^{-1}(q)\}$ for every $p,q\in\nat$.
\end{itemize}
Let $\wtilde\varphi:F_\infty\to\wtilde H$ be the unique homomorphism extending $\varphi$. We claim that $M=\ker\wtilde\varphi$ is a good choice. First, $M\in\cald$ follows from (2) and the fact that $\wtilde\varphi$ is a surjective homomorphism, and $f(M)=H$ follows from (3). For $M\in\calu$, note that for any $w\in\{u_0,\ldots,u_{k-1},v_0,\ldots,v_{l-1}\}$ we have
$$w\in M\iff \wtilde\varphi(w)=i_e\text{ holds in }\wtilde H\iff w(\varphi(x_0),\ldots,\varphi(x_{K-1}))=i_e\text{ holds in }\wtilde H\iff$$
$$w(\alpha_{N_0}(x_0),\ldots,\alpha_{N_0}(x_{K-1}))=i_e\text{ holds in }\wtilde H\iff w'=i_e\text{ holds in }\wtilde H\iff w\in\{u_0,\ldots,u_{k-1}\},$$
where the last equivalence holds by $H\in\calv$. $\blacksquare$
\end{proof}

\begin{cor}\label{c.main}
A group property is generic in $\calg$ if and only if it is generic in $\calm$.
\end{cor}

We prove a more general version, which is also worth recording.

\begin{cor}\label{c.main_general}
Let $\cals\subseteq\calg$ be any $G_\delta$ group property such that $\cals^*$ is $G_\delta$ in $\calm$. Then a group property $\calp\subseteq\cals$ is generic in $\cals$ if and only if $\calp^*$ is generic in $\cals^*$. (Recall Notation~\ref{n.associated_group_prop}.)
\end{cor}

\begin{proof}
First, note that $(\star)_{\cals^*}$ holds. By Lemma~\ref{l.open_dense}, $\cald\cap\cals^*$ is a Polish subspace in $\cals^*$. Observe that for the map $f$ given by Theorem~\ref{t.main}, $f|_{\cald\cap\cals^*}:\cald\cap\cals^*\to\cals$ is also surjective, continuous, and open. Thus, by Proposition~\ref{p.open_cont_surj}, a group property $\calp\subseteq\cals$ is generic in $\cals$ if and only if $f_{\cald\cap\cals^*}^{-1}(\calp)=\cald\cap\calp^*$ is generic in $\cald\cap\cals^*$. Since $\cald\cap\cals^*$ is comeager in $\cals^*$ by Lemma~\ref{l.open_dense}, this proves the corollary.
\end{proof}

\section{Applications}\label{s.applications}

\subsection{Countable groups}\label{ss.countable_groups}

The following results were proved in \cite{goldbring2023generic}, \cite{elekes2021generic}, and \cite{elekes2024generic}.
\begin{enumerate}
    \item[(G1)]\cite[Thm~1.1.6]{goldbring2023generic} Every isomorphism class is meager in $\calg$.
    \item[(G2)] \cite[Lemma~5.2.7]{goldbring2023generic} Algebraic closedness is a generic property in $\calg$.
    \item[(G3)] \cite[Prop~5.1.1]{goldbring2023generic} (0-1 law) Any Baire-measurable group property $\calp\subseteq\calg$ is either meager or comeager.
    \item[(G4)] \cite[Cor~5.9]{elekes2021generic} There exists a comeager elementary equivalence class in $\calg$.
    \item[(G5)] \cite[Thm~3.17]{elekes2024generic} For any group $H$, the set $\{G\in\calg:\ H\text{ embeds into G}\}$ is comeager in $\calg$ if and only if $H$ is countable and each finitely generated subgroup of $H$ has solvable word problem.
\end{enumerate}
Corollary~\ref{c.main} gives us the same results for the space of marked groups.

\begin{theorem}\label{t.generic_marked_groups}
Statements (G1)-(G5) hold when $\calg$ is replaced by the space of marked groups $\calm$, and the set $\{G\in\calg:\ H\text{ embeds into G}\}$ is replaced by $\{N\in\calm:\ H\text{ embeds into }F_\infty/N\}$.
\end{theorem}

\begin{remark}\label{r.osin}
D.~Osin \cite{osin2021topological} studied the space of \emph{finitely generated} marked groups and proved a characterization of closed subspaces of the space of finitely generated marked groups in which a 0-1 law holds for properties defined by $\call_{\omega_1,\omega}$-sentences. Using this characterization, Osin decided for various spaces of hyperbolic lacunary groups whether they contain a comeager elementary equivalence class.
\end{remark}

\subsection{Generic properties in various subspaces}

Consider the following group properties:
\begin{enumerate}
    \item abelian,
    \item torsion-free abelian,
    \item torsion-free,
    \item amenable,
    \item has no $F_2$ subgroups,
    \item has the unique product property (UPP),
    \item left orderable,
    \item locally indicable,
    \item biorderable,
    \item obeys the law $w(x_0,\ldots,x_{n-1})=e_G$ (for a fixed word $w$).
\end{enumerate}
\begin{remark}\label{r.G_delta_properties}
It is clear that abelian groups form a closed subspace in $\calg$. It was proved in \cite{goldbring2023generic} that properties (3)-(10) are $G_\delta$ in $\calg$. To prove that $\cals^*$ is $G_\delta$ in $\calm$ where $\cals\subseteq\calg$ is defined by one of properties (1)-(10) is, in each case, either easy or doable by a straightforward adaptation of the arguments in \cite{goldbring2023generic}.
\end{remark}

Recall from \cite[page 6250]{goldbring2023generic} that
$$\text{biorderable $\implies$ locally indicable $\implies$ left orderable $\implies$ has the UPP $\implies$ torsion-free.}$$
The following genericity results concerning various subspaces of $\calg$ were proved in \cite{goldbring2023generic, elekes2024generic, darji2023generic}.
\begin{enumerate}
    \item[(S1)] \cite[Cor~3.25]{elekes2024generic} In the subspace $\cala$ of abelian groups, the isomorphism class of the group $\ds\bigoplus_{i\in\nat}(\rat/\integer)$ is comeager.
    \item[(S2)] \cite[Thm~4.3]{darji2023generic} In the subspace of torsion-free abelian groups, the isomorphism class of $(\rat,+)$ is comeager.
    \item[(S3)] \cite[Cor~1.1.7]{goldbring2023generic} In the subspace of left orderable groups, there is no comeager isomorphism class.
    \item[(S4)] \cite[Cor~1.3.1]{goldbring2023generic} The generic group without $F_2$ subgroup is nonamenable.
    \item[(S5)] \cite[Cor~1.3.1]{goldbring2023generic} The generic amenable group is not elementary amenable.
    \item[(S6)] \cite[Cor~1.3.6]{goldbring2023generic} The generic torsion-free group does not have the UPP.
    \item[(S7)] \cite[Cor~1.3.6]{goldbring2023generic} The generic group with the UPP is not left orderable.
    \item[(S8)] \cite[Cor~1.3.3]{goldbring2023generic} The generic left orderable group is not locally indicable.
    \item[(S9)] \cite[Cor~1.3.3]{goldbring2023generic} The generic locally indicable group is not biorderable.
    \item[(S10)] \cite[Cor~5.4.4]{goldbring2023generic} For any closed nonamenable word $w$, the generic group satisfying the law $w=e$ is nonamenable.
\end{enumerate}

\begin{theorem}\label{t.subspaces_generic_prop}
Statements (S1)-(S10) also hold for the corresponding subspaces of $\calm$.
\end{theorem}

In the abelian case, for example, this means that (S1) holds when $\cala$ is replaced by $\cala'$ (see Subsection~\ref{ss.marked_abelian}). Note that $\cala'\supsetneqq \cala^*$ since the former contains finite groups.

\begin{proof}
Let $\cals$ be any subspace of $\calg$ defined by one of the properties (1)-(10). Let $\cals'$ be the subspace of $\calm$ defined by the same property. It is straightforward to check that $(\star)_{\cals'}$ holds. (In the case of property (10), we need to assume $\cals\neq\emptyset$.) By Lemma~\ref{l.open_dense}, $\cald\cap\cals'$ is comeager in $\cals'$. Since $\cald\cap\cals'\subseteq\cals^*\subseteq\cals'$ follows from the definitions, $\cals^*$ is also comeager in $\cals'$. Hence it suffices to prove (S1)-(S10) for $\cals^*$. Thus, Corollary~\ref{c.main_general}, Remark~\ref{r.G_delta_properties}, and the above theorems stating (S1)-(S10) for subspaces of $\calg$ conclude the proof.
\end{proof}

\subsection{Compact metrizable abelian groups}

In this subsection, we give new proofs of \cite[Thm~4.7]{elekes2024generic} and \cite[Thm~3.7]{darji2023generic}. Let
$$\cala=\left\{G\in\calg:\ \wtilde G\text{ is abelian}\right\}.$$
It is easy to check that $\cala$ is closed in $\calg$. Recall from Subsection~\ref{ss.marked_abelian} that
$$\cala''=\{N\subseteq F:\ N\leq F\},$$
where $F=\bigoplus_{i\in\nat}\integer$, and
$$\cala'=\{N\in\calm:\ F_\infty/N\text{ is abelian}\}$$
are closed subspaces of $2^F$ and $\calm$ respectively. Recall from Subsection~\ref{ss.compact_metriz_ab} the space $\cals(\TT)$ of compact metrizable abelian groups. Figure~\ref{f.abelian} illustrates the relationship between these spaces.

\begin{figure}[ht]
\centering
\includegraphics[scale=0.69]{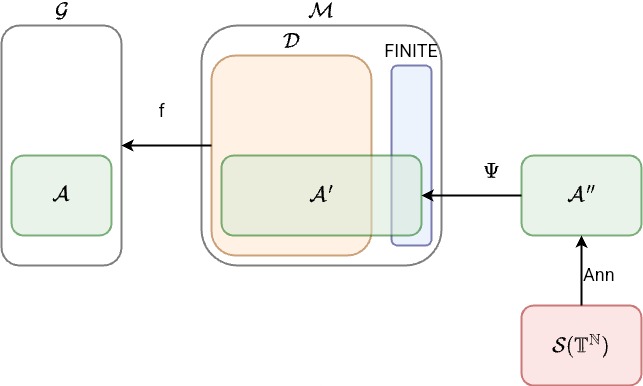}
\caption{}
\label{f.abelian}
\end{figure}

\vspace{1cm}

The following two theorems were proved separately in \cite{elekes2024generic}. We will use Corollary~\ref{c.main_general} to show that Theorem~\ref{t.generic_cpt_metriz_ab} follows directly from Theorem~\ref{t.generic_ctbl_ab}

\begin{theorem}\label{t.generic_ctbl_ab}
\cite[Cor~3.25]{elekes2024generic}
The isomorphism class of the group
$$A=\bigoplus_{i\in\nat}(\rat/\integer)$$
is comeager in $\cala$.
\end{theorem}

\begin{theorem}\label{t.generic_cpt_metriz_ab}
\cite[Thm~4.7]{elekes2024generic}
The isomorphism class of the group
$$Z=\prod_{p\text{ prime}}(Z_p)^\nat,$$
where $Z_p$ denotes the group of $p$-adic integers, is comeager in $\cals(\TT)$.
\end{theorem}

\begin{theorem}\label{t.cpt_metriz_ab_transfer}
Theorem~\ref{t.generic_cpt_metriz_ab} follows
from Theorem~\ref{t.generic_ctbl_ab}.
\end{theorem}

\begin{proof}
It is easy to check that $\cala^*$ is $G_\delta$ in $\calm$. Thus, by Theorem~\ref{t.generic_ctbl_ab} and Corollary~\ref{c.main_general}, the isomorphism class of $A$ is comeager in $\cala^*$. Since $(\star)_{\cala'}$ clearly holds and $\cald\cap\cala'\subseteq\cala^*$, Lemma~\ref{l.open_dense} shows that $\cala^*$ is comeager in $\cala'$, hence the isomorphism class of $A$ is comeager in $\cala'$ as well. By Proposition~\ref{p.abelian_equiv}, it follows that the isomorphism class of $A$ is comeager in $\cala''$ as well. Now Theorem~\ref{t.pontryagin_homeo} shows that the isomorphism class of $\what A$ is comeager in $\cals(\TT)$. However, $\what A=Z$ by Proposition~\ref{p.pontryagin} (3).
\end{proof}

Recall from Subsection~\ref{ss.compact_metriz_ab} the space $\calc(\TT)$ of connected compact metrizable abelian groups. Let
$$\calt\calf=\{G\in\cala:\ G\text{ is torsion-free}\},\quad \calt\calf'=\{N\in\cala':\ F_\infty/N\text{ is torsion-free}\},\quad\text{and}$$
$$\calt\calf''=\{N\in\cala'':\ F/N\text{ is torsion-free}\}.$$
The following two theorems were proved separately in \cite{darji2023generic}. We will use Corollary~\ref{c.main_general} to show that Theorem~\ref{t.solenoid} follows directly from Theorem~\ref{t.torsion_free}.

\begin{theorem}\label{t.torsion_free}
\cite[Thm~4.3]{darji2023generic}
The isomorphism class of $(\rat,+)$ is comeager in $\calt\calf$.
\end{theorem}

\begin{theorem}\label{t.solenoid}
\cite[Thm~3.7]{darji2023generic}
The isomorphism class of the universal solenoid is comeager in $\calc(\TT)$.
\end{theorem}

\begin{theorem}
Theorem~\ref{t.solenoid} follows from Theorem~\ref{t.torsion_free}.
\end{theorem}

\begin{proof}
It is easy to check that $\calt\calf$ is $G_\delta$ in $\calg$ and $\calt\calf^*$ is $G_\delta$ in $\calm$. Thus, by Theorem~\ref{t.torsion_free} and Corollary~\ref{c.main_general}, the isomorphism class of $(\rat,+)$ is comeager in $\calt\calf^*$. Since $(\star)_{\calt\calf'}$ clearly holds and $\cald\cap\calt\calf'\subseteq\calt\calf^*$, Lemma~\ref{l.open_dense} shows that $\calt\calf^*$ is comeager in $\calt\calf'$, hence the isomorphism class of $(\rat,+)$ is comeager in $\calt\calf'$ as well. By Proposition~\ref{p.abelian_equiv}, the isomorphism class of $(\rat,+)$ is comeager in $\calt\calf''$ as well. Now Theorem~\ref{t.pontryagin_homeo} and Proposition~\ref{p.pontryagin} (5) show that the isomorphism class of the dual of $(\rat,+)$ is comeager in $\calc(\TT)$. This concludes the proof because the dual of $(\rat,+)$ is the universal solenoid by Proposition~\ref{p.pontryagin} (4). 
\end{proof}

\section*{Acknowledgment}

I would like to thank Slawomir Solecki for encouraging me to write up these results in a short paper, and Anush Tserunyan for drawing my attention to the paper \cite{CHEN_2025}. I am especially grateful to Ruiyuan Chen for a helpful and inspiring discussion and also for allowing me to include his arguments in the appendix.

\section*{Funding declaration}

The author was supported by the National Research, Development and Innovation Office -- NKFIH, grant no.~146922 and the Hungarian Academy of Sciences Momentum Grant no.~2022-58.

\printbibliography

\appendix

\section*{Appendix}

The arguments presented in this appendix are due to Ruiyuan Chen. Based partially on personal communication, the text itself was written by the author, who takes full responsibility for mistakes.

First, let us outline very briefly how Corollary~\ref{c.main_intro} can be derived from results in \cite{CHEN_2025}. Let $\cali$ denote the space of all isomorphisms between structures in $\calg$ and structures in $\calm$. By Examples~6.1 and 6.12 in \cite{CHEN_2025}, $\calg$ and $\calm$ can be viewed as étale bundles of structures with $\Sigma_1$ saturations. It follows easily (see \cite[Lemma~5.18]{CHEN_2025}) that the domain $\dom:\cali\to\calg, \varphi\mapsto \dom(\varphi)$ and codomain $\cod:\ \cali\to\calm, \varphi\mapsto \cod(\varphi)$ maps are open and continuous onto $\dom(\cali)$ and $\cod(\cali)$ respectively. Here $\dom(\cali)=\calg$, and $\cod(\cali)$ is a comeager subset of $\calm$. Using Proposition~\ref{p.open_cont_surj}, we conclude that generic properties can be transferred between $\calg$ and $\calm$.

Now we present a nice, more self-contained argument that provides an alternative proof for Corollary~\ref{c.main_intro}. Readers not familiar with étale structures may find this proof more accessible.

Let $\calx$ denote any of the spaces $\calg$ and $\calm$, and let $\calp\subseteq\calx$ be a group property.

\begin{theorem*}
The group property $\calp$ is generic in $\calx$ if and only if there is a $\Pi_2$ sentence $\Phi$ in the logic $L_{\omega_1,\omega}$ (see \cite[Sec~16.C]{kechris2012classical}) based on the language $L$ of groups such that the following hold.

(A) $\Phi\implies\calp$. More precisely, the set $\calp_\Phi\subseteq\calx$ defined by $\Phi$ is a subset of $\calp$.

(B) For every $\Sigma_1$ sentence $\Psi$ in $L_{\omega_1,\omega}$ that is satisfied by at least one countably infinite group, the conjunction $\Psi\land\Phi$ is also satisfied by at least one countably infinite group.
\end{theorem*} 

Note that Corollary~\ref{c.main_intro} follows immediately from this theorem.

\begin{proof}
First, we will prove the following.

\textbf{Claim~1.} A group property $\calp$ is generic in $\calx$ if and only if it contains a set $\calb$ that is \emph{isomorphism-invariant} and dense $G_\delta$ in $\calx$.

\textbf{Proof.} It suffices to prove the \emph{only if} part. Let $\calp\subseteq\calx$ be a comeager group property. (Thus $\calp$ is invariant by assumption.) Pick any dense $G_\delta$ set $\calb_0\subseteq\calx$ such that $\calb_0\subseteq\calp$.

\emph{Case~1.} $\calx=\calg$. Recall that there is a natural continuous action $\alpha:S_\infty\times\calg\to\calg$, namely, we define $\sigma G$ by pushing forward the group operation $G$ along the bijection $\sigma$. We claim that the so-called Vaught transform (see \cite[Def.~16.2]{kechris2012classical})
$$\calb={\calb_0}^*=\{G\in\calg:\ \{\sigma\in S_\infty:\ \sigma G\in \calb_0\}\text{ is comeager in }S_\infty\}$$
is a good witness. The facts that $\calb\subseteq\calp$ and $\calb$ is isomorphism-invariant are clear from the definition. It is $G_\delta$ since the Vaught transform of a $\Pi^0_\alpha$ set is $\Pi^0_\alpha$ by \cite[Ex~22.23]{kechris2012classical}. (Also, this is easy to prove for $G_\delta$ sets.) It remains to prove that $\calb$ is comeager in $\calg$. Let $\calc_0=\alpha^{-1}(\calb_0)\subseteq S_\infty\times\calg$.
%Write the set $\calc_0=\alpha^{-1}(\calb_0)\subseteq S_\infty\times\calg$, which is $G_\delta$ since $\alpha$ is continuous, as an intersection of open sets: $\calc_0=\bigcap_{n\in\nat}\calu_n$. Let $\{\calv_k:\ k\in\nat\}$ be any countable basis for $S_\infty$ consisting of nonempty sets. Now $G\in\calb$ if and only if each open set $(\calu_n)^G$ is dense in $S_\infty$, where $(\calu_n)^G$ is the horizontal section $\{\sigma\in S_\infty:\ (\sigma,G)\in\calu_n\}$. Thus
%$$G\in\calb\iff \forall n, k\in\nat\ ((\calu_n)^G\cap\calv_k\neq\emptyset)\iff G\in\bigcap_{n,k\in\nat}\proj_\calg(\calu_n\cap(\calv_k\times\calg)),$$
%where $\proj_\calg$ is the projection to $\calg$. Hence $\calb$ is indeed $G_\delta$.
Since the vertical section $(\calc_0)_\sigma=\sigma^{-1}(\calb_0)$ is comeager for every $\sigma\in S_\infty$, it follows by the Kuratowski--Ulam theorem that $\calb$ is also comeager.

\emph{Case~2.} $\calx=\calm$. This is similar to Case 1, but we need to use a groupoid action and an appropriate generalization of the Kuratowski--Ulam theorem. Here we only sketch the proof. The interested reader may consult the notes \cite{CHEN_NOTES_2025} to fill in the missing details and precise definitions.

The isomorphism groupoid $\cali$ of $\calm$ is the set
$$\left\{(M,N,f)\in\calm^2\times 2^{{F_\infty}^2}:
\begin{array}{l}
\forall u,u',v,v'\in F_\infty\ (uM=u'M\land vN=v'N\implies(f(u,v)=f(u',v')))\land\\
\forall u\in F_\infty\ \exists v,w\in F_\infty\ (f(u,v)=1\land f(w,u)=1)\land\\
\forall u,v,v'\in F_\infty\ (vN\neq v'N\implies (f(u,v)=0\lor f(u,v')=0))\land\\
\forall u,u', v\in F_\infty\ (uM\neq u'M\implies f(u,v)=0\lor f(u',v)=0)\land\\
\forall u,u',v,v',w\in F_\infty\ (f(u,v)=1\land f(u',v')=1\land f(uu',w)=1\implies\\
\implies wN=vv'N)
\end{array}\right\}$$
equipped with the subspace topology inherited from $\calm^2\times 2^{{F_\infty}^2}$. A triple $(M,N,f)$ encodes an isomorphism between $F_\infty/M$ and $F_\infty/N$ as a subset of $F_\infty\times F_\infty$ that is invariant under the product of the coset equivalence relations associated to $M$ and $N$. Note that $\cali$ is Polish. Let $\dom: \cali\to\calm, \dom(M,N,f)=M$ and $\cod:\cali\to\calm, \cod(M,N,f)=N$. It is clear from the definitions that the maps $\dom$ and $\cod$ are continuous and surjective. Let us prove that they are open. By symmetry, it suffices to check $\dom$. We need to introduce further notions.

Sets of form
$$\calw=\{(M,N,f)\in\cali:\ M\in\calu\land N\in\calv\land f(u_0,v_0)=1\land\ldots\land f(u_{m-1},v_{m-1})=1\}$$
with $\calu,\calv\subseteq\calm$ open, $l\in\nat$, and $u_0,v_0,\ldots,u_{m-1},v_{m-1}\in F_\infty$ constitute a basis for $\cali$.  We define the space of $m$-pointed marked groups as $\calm_m=(\calm\times {F_\infty}^m)/\sim$, where $(M,(u_0,\ldots,u_{m-1}))\sim (N,(v_0,\ldots,v_{m-1}))$ if and only if $M=N$ and $u_iM=v_iM$ for each $i<m$. Let $\varrho_m:\calm\times {F_\infty}^m\to\calm_m$ denote the quotient map. It is easy to check that sets of the form $\varrho_m(\calu\times\{(u_0,\ldots,u_{m-1})\})$ with $\calu\subseteq\calm$ open and $u_0,\ldots,u_{m-1}\in F_\infty$ constitute a basis for $\calm_m$. The isomorphism groupoid $\cali$ acts naturally on $\calm_m$: the element $(M,N,f)$ maps the class of $(M,(u_0,\ldots,u_{m-1}))$ to the class of $(N,(v_0,\ldots,v_{m-1}))$, where the $v_i$ are such that $f(u_i,v_i)=1$ for each $i<m$. Let $\pi:\calm_m\to\calm$ denote the ``projection''. That is, $\pi$ maps the class of $(M,(u_0,\ldots,u_{m-1}))$ to $M$. Note that this is an open map.
Now we can write the set $\dom(\calw)$ as
$$\pi(\varrho_m(\calu\times\{(u_0,\ldots,u_{l-1})\})\cap(\cali\cdot\varrho_m(\calv\times\{(v_0,\ldots,v_{l-1})\}))),$$
where $\cali\cdot\varrho_m(\calv\times\{(v_0,\ldots,v_{l-1})\})$ is the isomorphism saturation of $\varrho_m(\calv\times\{(v_0,\ldots,v_{l-1})\})$, which is open by \cite[Example~6.12]{CHEN_2025} (see also \cite[Lemma]{CHEN_NOTES_2025}). Since the map $\pi$ is open, we conclude that $\dom(\calw)$ is open.

Let $\calc_0=\cod^{-1}(\calb_0)$ and
$$\calb=\{M\in\calm:\ \calc_0\text{ is comeager in } \dom^{-1}(M)\}.$$
Then $\calb\subseteq\calp$ is clear from the definition, and it is easy to prove that $\calb$ is isomorphism-invariant. Let $\{\calv_k:\ k\in\nat\}$ be any countable basis for $\cali$. Similarly to Case 1, $\calc_0$ is $G_\delta$ in $\cali$, hence we can write it as an intersection of open sets: $\calc_0=\bigcap_{n\in\nat}\calu_n$. Now $M\in\calb$ if and only if each open set $\calu_n$ is dense in $\dom^{-1}(M)$. Thus
$$M\in\calb\iff \forall n, k\in\nat\ (M\in\dom(\calv_k)\implies M\in\dom(\calu_n\cap\calv_k)),$$
which shows that $\calb$ is $G_\delta$ since the map $\dom$ is open. To see that $\calb$ is comeager in $\calm$ first note that $\calc_0$ is comeager in $\cali$ by Proposition~\ref{p.open_cont_surj}. Then apply the following generalization of the Kuratowski--Ulam theorem to the map $\dom:\cali\to\calm$.

\begin{theorem*}\label{t.general_KU}
\cite[Thm A.1]{MELLERAY_TSANKOV_2013} Let $X$, $Y$ be Polish spaces and $f:X\to Y$ be a continuous, open map. Let $A$ be a Baire measurable subset of $X$. Then the following are equivalent:
\begin{enumerate}
    \item[(A)] $A$ is comeager in $X$;
    \item[(B)] for comeager many $y\in Y$ the set $A\cap f^{-1}(y)$ is comeager in $f^{-1}(y)$.
\end{enumerate}
\end{theorem*}
This concludes the proof of Claim~1. It remains to prove the following.

\textbf{Claim~2.} For a set $\cals\subseteq\calx$ the following are equivalent.

(1) It is an isomorphism-invariant dense $G_\delta$ subset of $\calx$.

(2) It is defined by a $\Pi_2$ sentence $\Phi$ in the logic $L_{\omega_1\omega}$ such that condition (B) holds.

\textbf{Proof.} In the case $\calx=\calg$, the fact that isomorphism-invariant $G_\delta$ sets are exactly the sets defined by $\Pi_2$ sentences in $L_{\omega_1\omega}$ is a special case of the level-by-level version \cite[Ex~22.24]{kechris2012classical} of the classical Lopez-Escobar theorem \cite[Thm~16.8]{kechris2012classical}, which is due to Vaught \cite{VAUGHT_74}. In the case $\calx=\calm$, the same fact is a special case of \cite[Thm~10.2]{CHEN_2025}, which is stated generally for étale structures. Readers not familiar with étale structures may consult the notes \cite{CHEN_NOTES_2025} about the Lopez-Escobar theorem for marked structures.

To verify $(1)\implies (2)$ fix any set $\cals\subseteq\calx$ that satisfies $(1)$. By the first paragraph, it is defined by a $\Pi_2$ sentence $\Phi$ in $L_{\omega_1\omega}$. Let $\Psi$ be a $\Sigma_1$ sentence in $\inflogic$ that is satisfied by a countably infinite group. Then $\Psi$ defines a nonempty open set $\calv$ in $\calx$. Since countably infinite groups form a comeager set in $\calx$, the intersection $\cals\cap\calv$ contains a countably infinite group, which is a model of $\Phi\land\Psi$.

For $(2)\implies (1)$, fix any set $\cals\subseteq\calx$ that satisfies (2), and any nonempty open set $\calu\subseteq\calx$. Since $\cals$ is isomorphism-invariant, we may assume that $\calu$ is also isomorphism-invariant. (Here we again use the crucial fact that the saturation of an open set is open.) Then there is a $\Sigma_1$ sentence $\Psi$ in $L_{\omega_1\omega}$ that defines $\calu$. Since countably infinite groups form a dense set in $\calx$, by (2), there is a countably infinite group that satisfies $\Phi\land\Psi$, witnessing that $\cals\cap\calu$ is nonempty.

The proof of the theorem is complete.
\end{proof}

\end{document}

%% file: author.tex
\author{Tamás~Kátay}
\address{\normalfont (TK) Alfréd Rényi Institute of Mathematics, Budapest, Hungary}
\email{13heted@gmail.com}